\documentclass[12pt]{amsart}
\usepackage{amsmath,amssymb,amsfonts,amsthm,amsopn}
\usepackage{graphics}

 \def\Spnr{Sp(d,\R)}

\newcommand{\tfa}{time-frequency analysis}

\newcommand{\stft}{short-time Fourier transform}

\newcommand{\tf}{time-frequency}

\newcommand{\fif}{if and only if}
\newcommand{\tfs}{time-frequency shift}

\newcommand{\modsp}{modulation space}
\newcommand{\psdo}{pseudodifferential operator}

\newtheorem{theorem}{Theorem}[section]
\newtheorem{lemma}[theorem]{Lemma}
\newtheorem{corollary}[theorem]{Corollary}
\newtheorem{proposition}[theorem]{Proposition}
\newtheorem{definition}[theorem]{Definition}
\newtheorem{cor}[theorem]{Corollary}
\newtheorem{remark}[theorem]{Remark}

\newcommand{\beqa}{\begin{eqnarray*}}
\newcommand{\eeqa}{\end{eqnarray*}}

\DeclareMathOperator*{\essupp}{ess\,sup\,}

\newcommand{\field}[1]{\mathbb{#1}}
\newcommand{\bR}{\field{R}}        
\newcommand{\bC}{\field{C}}        
\def\la{\lambda}

\def\cS{\mathcal{S}}

\def\cH{\mathcal{H}}

\def\cM{\mathcal{M}}

\def\cA{\mathcal{A}}

\def\cC{\mathcal{C}}

\def\cX{\mathcal{X}}

\def\rd{\bR^d}

\def\rdd{{\bR^{2d}}}

\def\lrd{L^2(\rd)}

\def\mif{M^{\infty, 1}}
\def\mifs{M^{\infty, 1}_{1\otimes v_s}}

\def\intrd{\int_{\rd}}
\def\intrdd{\int_{\rdd}}

\def\R{\right)}

\def\<{\left<}
\def\>{\right>}

\def\inv{^{-1}}

\def\mv1{M_v^1}

\def\phas{(x,\o )}
\def\mn{(m,n)}
\def\mn'{(m',n')}

\def\Spnr{Sp(d,\R)}

\def\o{\eta}

\def\R{\mathbb{R}}
\def\Ren{\mathbb{R}^d}
\def\Renn{\mathbb{R}^{2d}}

\def\sch{\mathcal{S}}

\def\Fur{\mathcal{F}}

\def\f{\varphi}

\def\Sn2{S_{2}(L^{2}(\Ren))}
\def\S1{S_{1}(L^{2}(\Ren))}
\def\sig00{\sigma_{0,0}}

\def\la{\langle}
\def\ra{\rangle}

\newcommand{\A}{\mathcal{A}}

\begin{document}
\begin{abstract} It is well known that the  matrix of a  metaplectic
  operator with respect to phase-space shifts  is concentrated along the
  graph of a linear symplectic map. We show that the algebra generated
  by metaplectic operators and  by \psdo s in a Sj\"ostrand class
  enjoys the same decay properties. We study the behavior of these
  generalized metaplectic operators and represent them by Fourier
  integral operators. Our main result shows that the one-parameter
  group generated by a Hamiltonian  operator with a potential in the
  Sj\"ostrand class consists of generalized metaplectic operators. As
  a consequence, the Schr\"odinger equation preserves the phase-space
  concentration, as measured by \modsp\ norms.
\end{abstract}

\title[Generalized Metaplectic Operators]{Generalized Metaplectic Operators and the Schr\"odinger
  Equation with  a  Potential in the  Sj\"ostrand Class}

\author{Elena Cordero}
\address{Universit\`a di Torino, Dipartimento di Matematica, via Carlo Alberto 10, 10123 Torino, Italy}
\email{elena.cordero@unito.it}
\author{Karlheinz Gr\"ochenig}
\address{Faculty of Mathematics,
University of Vienna, Nordbergstrasse 15, A-1090 Vienna, Austria}
\email{karlheinz.groechenig@univie.ac.at}
\author{Fabio Nicola}
\address{Dipartimento di Scienze Matematiche,
Politecnico di Torino, corso Duca degli Abruzzi 24, 10129 Torino,
Italy}
\email{fabio.nicola@polito.it}
\author{Luigi Rodino}
\address{Universit\`a di Torino, Dipartimento di Matematica, via Carlo Alberto 10, 10123 Torino, Italy}
\email{luigi.rodino@unito.it}
\thanks{This work was completed with the support of the Erwin
  Schr\"odinger  Institute for Mathematical Physics,
  Vienna, Austria,  and the Center for Advanced Study in Oslo.  K.\ G.\ was
  supported in part by the  project P22746-N13  of the
Austrian Science Foundation (FWF)}

\maketitle

\section{Introduction}

Metaplectic operators describe linear symplectic transformations of
phase space and  arise as intertwining  operators of the Schr\"odinger
representation of the Heisenberg group.   They solve  the 
free Schr\"odinger equation  and  the Schr\"odinger of the 
harmonic oscillator~\cite{folland89}. The generic metaplectic operator
can be represented by a Fourier integral operator with a quadratic
phase function and constant symbol. In this paper we introduce  a general
class of Fourier integral operators with quadratic phase function and
study the fundamental properties, such as boundedness, composition and
invertibility. Our main result shows that these generalized
metaplectic operators solve the Schr\"odinger equation with a
quadratic Hamiltonian and a bounded, but not necessarily smooth
perturbation. This is a far-reaching generalization of a result of Weinstein~\cite{Wein85}.

To fix notation, we write a point in phase space (in \tf\ space) as
$z=(x,\eta)\in\rdd$, and  the corresponding phase-space shift (\tfs )
acts on a function or distribution  as 
\begin{equation}
  \label{eq:kh25}
  \pi (z)f(t) = e^{2\pi i \eta t} f(t-x) \, .
\end{equation}
If $\A $ is a symplectic matrix on $\rdd $, i.e., $\A ^T J \A  = J$,
where
$$
J=\begin{pmatrix} 0&-I_d\\I_d&0\end{pmatrix}
\, 
$$
yields the standard symplectic form on $\rdd $, then the 
metaplectic operator $\mu (\cA )$ is  defined by the intertwining
relation
\begin{equation}\label{metap}
\pi (\cA z) = c_\cA  \, \mu (\cA ) \pi (z) \mu (\cA )\inv  \quad  \forall
z\in \rdd \, ,
\end{equation}
with a phase factor $c_\cA \in \bC , |c_{\cA } | =1$  (for details, see e.g. \cite{folland89,Gos11}).

Our point of departure is the decay of the kernel or matrix of the
metaplectic operator with respect to the set of \tfs s (phase-space
shifts) $\pi (z)$. If $\A \in \mathrm{Sp}(d,\R)$ and $g\in \cS (\rd
)$, then for $N\geq 0$ there exist a $C_N >0$ such that 
\begin{equation}
  \label{eq:kh21}
|\langle \mu (\A ) \pi (z)g, \pi (w) g\rangle |\leq C_N\langle w-\A z
\rangle ^{-N},\qquad w,z\in\R^{2d}. 
\end{equation}

Continuing our investigations in  \cite{Wiener}, we define
generalized metaplectic operators  by the decay of the
corresponding kernel as follows. To describe general decay conditions,
we use the polynomial weights $v_s(z) = (1+|z|)^s $ for $s\in \bR $
and the weighted $L^1$-space defined by the norm $\|f\|_{L^1_{v_s}} =
\|f v_s\|_1$. 

\begin{definition}\label{def1.1} Given $\cA \in \Spnr $,
  $g\in\cS(\rd)$,  and $s\geq0$, we say that  a
 linear operator $T:\cS(\rd)\to\cS'(\rd)$ is in the
class $FIO(\cA,v_s)$ if there exists a function $H\in L^1_{v_s}(\rdd)$
such that the kernel of $T$ with respect to \tfs s satisfies the decay
condition
\begin{equation}\label{asterisco}
|\langle T \pi(z) g,\pi(w)g\rangle|\leq H(w-\cA z),\qquad \forall w,z\in\rdd.
\end{equation}
\end{definition}
The union 
\[
FIO(Sp(d,\R),v_s)=\bigcup_{\mathcal{A}\in Sp(d,\R)} FIO(\mathcal{A},v_s)
\]
is then the  class of \emph{generalized metaplectic operators}. 

In the first part of the paper we study the fundamental properties of
generalized metaplectic operators. These concern the boundedness, the
composition of generalized metaplectic operators, properties of the
inverse operator, and alternative representations. Briefly, the main
theorem can be summarized as follows.  

\begin{theorem}[Main properties of generalized metaplectic operators] \label{tc1}

  (i) The definition of $FIO(\mathcal{A},v_s)$ is 
  independent of the window $g\in\cS(\rd)$. 

(ii) An operator $T\in  FIO(\mathcal{A} ,v_s)$ is bounded on every \modsp\
$M^p_{v_s} (\rd )$ for $1\leq p \leq \infty $ and $s\in \bR
$. (See~\eqref{defmod} for the precise definition.)

(iii)  If $T_i\in FIO(\mathcal{A}_i,v_s)$, $i=1,2$, then $T_1 T_2\in
FIO(\mathcal{A}_1\mathcal{A}_2,v_s)$. 

(iv)  If $T\in FIO(\mathcal{A},v_s)$ is invertible on $L^2(\rd)$,
then $T^{-1}\in FIO(\mathcal{A}^{-1},v_s)$. 
\end{theorem}

In short,  the class of the generalized metaplectic operators
$FIO(Sp(d,\R),v_s)$, $s\geq0$, is a Wiener sub-algebra of
$B(L^2(\rd))$. 

Definition~\ref{def1.1} is convenient, but it does not provide a
concrete  formula for  generalized metaplectic
operators. For a more analytic description we use Fourier integral
operators and \psdo s in the Weyl calculus  defined by 
\[
Tf(x)=\sigma^w(x,D)f=\int_{\rdd} e^{2\pi i(x-y)\eta}
\sigma\Big(\frac{x+y}{2},\eta\Big)f(y)\,dy\,d\eta \, .
\]
Hence, for the sake of simplicity we omit the Planck constant. In view of the main results of the paper, this is not a restricted setting.\par
The appropriate symbol class is the \modsp\ $\mifs (\rdd)$ defined by the
norm  consisting of all symbols on $\rdd $ such that 
\begin{equation}\label{otto}
\|\sigma \|_{\mifs } = \int_{\rdd}\sup_{z\in\rdd}|\langle \sigma,\pi(z,\zeta)g\rangle |v_s(\zeta)\, d\zeta<\infty.
\end{equation}

Then we have the following characterization of $FIO(\A , v_s)$.

\begin{theorem}[Representations of generalized metaplectic operators] \label{tc2}
(i) An operator $T$ is in $FIO(\A , v_s)$  \fif\ 
 there exist  symbols  $\sigma_1$ and $\sigma _2 \in M^{\infty,1}_{1\otimes v_s}(\rdd)$ such that
 \begin{equation}
   \label{eq:kh23}
T=\sigma_1^w(x,D)\mu(\A) =\mu(\A)\sigma_2^w(x,D).   
 \end{equation}

(ii) Let $
\A=\Big(\begin{smallmatrix} A&B\\C&D\end{smallmatrix}\Big)\in Sp(d,\R) $ be a symplectic
matrix with $ \det A\not=0 $ and define
the phase function $\Phi $ as 
\begin{equation}\label{fase}\Phi(x,\eta)=\frac12  x CA^{-1}x+
\eta  A^{-1} x-\frac12\eta  A^{-1}B\eta.
\end{equation}
Then $T \in FIO(\A , v_s)$ \fif\ $T$ can be written as a  type I
Fourier integral operator (FIO) in the form 
\begin{equation}
  \label{eq:kh22}
Tf(x)=\int_{\rd} e^{2\pi i\Phi(x,\eta)}\sigma(x,\eta)\hat{f}(\eta)\, d\eta,  
\end{equation}
\end{theorem}

In view of the factorization~\eqref{eq:kh23}   the class of
generalized metaplectic operators is just the algebra generated by the
metaplectic representation and the algebra of \psdo s with a symbol in
the \modsp \,$\mifs(\rdd) $. 
 
\begin{remark}\rm
  (i) In general, the composition of two type I Fourier integral
  operators  is not of type I, even if the phase is quadratic. 
  Definition~\ref{def1.1}  and Theorem~\ref{tc1} solve the 
  problem of constructing an   algebra of FIOs  in a trivial way. 

(ii) An  especially attractive feature of generalized metaplectic
operators  is the  decay of the associated kernel
off a subspace that is determined by the symplectic transformation of
phase space. We expect this fact  to
be a useful tool in the numerical analysis of generalized metaplectic
operators. See~\cite{fio3,fio-evolution}  for first results.   

(iii) Theorems~\ref{tc1} and \ref{tc2} are preceded by numerous
contributions in a similar spirit and hold on various levels of
generality.  

If $\A = \mathrm{Id}$, Theorem~\ref{tc1} goes back to
Sj\"ostrand~\cite{wiener30,wiener31} and describes the
properties of an important class of \psdo s with symbols in $\mifs $,
nowadays called Sj\"ostrand class. The \tfa\ of the Sj\"ostrand class
and its characterization by the off-diagonal decay \eqref{asterisco} was
developed in ~\cite{charly06,GR}.  Indeed, our initial motivation was
to study these results in the context of FIOs, and  our techniques rely heavily
on the \psdo\ calculus developed in~\cite{charly06}. 

The sparsity of genuine FIOs with a tame phase and a symbol in
$S^0_{0,0}$  with respect to \tfs s
was established in~\cite{fio1}. In fact, much of the later research tried
to relax the smoothness assumptions on the symbol. For instance, 
in \cite{Wiener} we considered FIOs with a tame phase and a more rigid
polynomial decay condition \eqref{eq:kh21} of the corresponding kernel. However, the
case of $L^1$-decay ~\eqref{asterisco} presented additional technical
difficulties, which we will solve in Section 3.

The boundedness of FIOs with a symbol in $\mif $ on $\lrd $ was
studied in \cite{wiener6,wiener7,wiener8,wiener9,fio1}. 
\end{remark}

Our  main contribution and  innovation is the  application of generalized  metaplectic
operators to the study of the Schr\"odinger equation with an initial
condition in a \modsp . We consider a Hamiltonian of the form 
$$
H = a^w (x,D)  + \sigma ^w(x,D)
$$
where $a$ is  a real-valued, quadratic, homogeneous polynomial on
$\rdd $ and $\sigma $ is a symbol in a \modsp\ $\mifs $. The standard
Hamiltonians are $a(x,\xi ) = |\xi| ^2$ with $a^w(x,D) = - \Delta $ and
$\sigma = 0$  for the free Schr\"odinger  equation and $a(x,\xi )  =
|\xi| ^2+|x|^2$ and $\sigma = 0$ for the harmonic oscillator. For these
cases the one-parameter evolution group is given by a metaplectic
operator $e^{itH}= \mu (\cA _t)$ for a one-parameter subgroup $t\mapsto \A _t$
of the symplectic group~\cite{folland89}. In this case the time evolution
leaves a large class of \modsp s invariant, as was shown in
~\cite{Benyi,wh}
for the free Schr\"odinger equation by studying Fourier multipliers on
\modsp s and in ~\cite{fio5} by checking the invariance properties of
\modsp s.

It is natural to conjecture that similar results hold for the
Schr\"odinger equation with potentials. The following result offers an
answer under rather general conditions and connects with generalized
metaplectic operators. 

\begin{theorem} \label{tc3}
  Let $a$ be  a real-valued, quadratic, homogeneous polynomial on
$\rdd $, $\sigma $  a symbol in a \modsp\ $\mifs (\rdd)$ and $H=
a^w(x,D) + \sigma ^w(x,D)$. Then for every $t\in \bR $ the
corresponding operator $e^{itH}$ is a generalized metaplectic
operator.
 
Consequently, the Schr\"odinger equation preserves the phase-space
concentration: if $ u(0,\cdot ) = f \in M^p_{v_s}(\rd)$, then also
$u(t,\cdot )
= e^{itH}f\in  M^p_{v_s}(\rd)$, for all $t\in\bR$. 
\end{theorem}

In fact, the precise formulation yields more information. If $t\mapsto
\A _t$ describes the solution of the classical equations of motion
with Hamiltonian $a(x,\xi )$ in phase-space, then there exists a symbol
$b_t \in \mifs $, such that  
\begin{equation}
  \label{eq:kh27}
 e^{itH }  = \mu (\A _t) b_t(x,D) \, .
\end{equation}

The main insight of Theorem~\ref{tc3}  is that the quality of
phase-space concentration  is
preserved by the physical time evolution, when \modsp\ norms are used
as  the natural measure of
phase-space concentration. This is  in contrast to the
standard theory with $L^p$-spaces. Indeed, it is well known that even the free propagator $e^{it\Delta}$ is not bounded in $L^p(\rd)$ except for $p=2$ (see e.g.\ \cite[pag. 74]{tao}).  This shows once more  that \modsp s are a useful
concept for the quantitative measurement of phase space
concentration. 
 
\begin{remark}\rm 
(i) Note that the perturbation $\sigma ^w(x,D)$ may be a genuine
\psdo\ with a non-smooth symbol, not just a multiplication
operator. 

(ii) The boundedness of $e^{itH } $ on \modsp s for $\sigma ^w(x,D) =
V(x)$ with second derivatives  in the H\"ormander class
$ S^0_{0,0}$ was treated in~\cite{kki4}, various estimates for the
Schr\"odinger propagator with Hamiltonian $H=-\Delta +|x|^2$ were
obtained in~\cite{cnr,kki1,kki2,kki3}, and  contributions to the
$L^p$-theory are contained in~\cite{js94,js95}.   

(iii)  For a special subclass $Z_n$ of the H\"ormander class
$S^0_{0,0}$, a factorization of the form~\eqref{eq:kh27} was found by
Weinstein~\cite{Wein85}. Subsequently,  integral representations of the evolution group with a
Hamiltonian $H=-\Delta +xA x+ V(x)$ for $V\in S^0_{0,0}$ were
obtained in~\cite{gz}. Theorem~\ref{tc3} is  much  more general,
and  it also  offers the
additional insight of how the evolution group is related to the
metaplectic representation.  

(iv)  If the Hamiltonian $a(x,\xi)$ is real-valued and satisfies
$\partial^\alpha_{x,\xi} a\in S^0_{0,0}$ for $|\alpha|\geq 2$, then,  by
classical results~\cite{wiener1,wiener3,wiener4,Wiener,
  tataru}, 
the Schr\"odinger propagator is a FIO of type I represented by \eqref{eq:kh22} for $t>0$
 sufficiently small.  By the results in~\cite{Wiener} this extends to all  $t\in\bR$  away from caustics.   \par 

\end{remark}

Theorem~\ref{tc3} contains as special cases several known results (cf.\ \cite{Benyi,fio5,fio1}) about the action of the Schr\"odinger propagator $e^{itH}$ on \modsp s. 

The paper is organized as follows: Section~2 contains the preliminary
notions from \tfa . In Section~3 we study the main properties of
generalized metaplectic operators. Section~4 contains the analysis of
the Schr\"odinger equation. Section~5 relates the generalized
metaplectic operators with Fourier integral operators. 

\vskip0.3truecm
\textbf{Notation.} We write $xy=x\cdot y$ for  the scalar product on
$\Ren$ and $|t|^2=t\cdot t$ for $t,x,y \in\Ren$.

The Schwartz class is denoted by
$\sch(\Ren)$, the space of tempered
distributions by  $\sch'(\Ren)$.   We
use the brackets  $\la f,g\ra$ to
denote the extension to $\sch '
(\Ren)\times\sch (\Ren)$ of the inner
product $\la f,g\ra=\int f(t){\overline
{g(t)}}dt$ on $L^2(\Ren)$. The Fourier
transform is normalized to be ${\hat
  {f}}(\o)=\Fur f(\o)=\int
f(t)e^{-2\pi i t\o}dt$.

Sometimes we will use 
\begin{equation}\label{tre}
T_x f(t)=f(t-x),\qquad M_\eta f(\eta)= e^{2\pi i \eta t} f(t),\quad t,x,\eta \in\rd.
\end{equation}
for the operators of translation and modulation acting on a function
or distribution  $f$. Their composition is then  the time-frequency shift
$\pi(z)=M_\eta T_x$.

We  write
$A \asymp B$  for the equivalence  $c^{-1}B\leq
A\leq c B$.

\vskip0.3truecm
\section{Preliminaries}
  We recall the basic
concepts  of \tfa\ and  refer the  reader to \cite{book} for the full
details.
\subsection{The Short-time Fourier Transform }
Consider a distribution $f\in\cS '(\rd)$
and a Schwartz function $g\in\cS(\rd)\setminus\{0\}$ (the so-called
{\it window}).
The short-time Fourier transform (STFT) of $f$ with respect to $g$ is $$V_gf (z) = \langle f, \pi (z)g\rangle, \quad z=(x,\xi)\in\rdd.$$
 The  \stft\ is well-defined whenever  the bracket $\langle \cdot , \cdot \rangle$ makes sense for
dual pairs of function or (ultra-)distribution spaces, in particular for $f\in
\cS ' (\rd )$ and $g\in \cS (\rd )$,  or for $f,g\in\lrd$.
\par

\subsection{Modulation and amalgam spaces}
Weighted modulation spaces measure the decay of the STFT on the time-frequency (phase space) plane and were introduced by Feichtinger in the 80's \cite{F1}.

\emph{Weight Functions.}  A weight function $v$ is submultiplicative if $ v(z_1+z_2)\leq v(z_1)v(z_2)$, for all $z_1,z_2\in\Renn.$  We shall work with the weight functions
\begin{equation} v_s(z)=\la z\ra^s=(1+|z|^2)^{\frac s 2},\quad s\in\R,
\end{equation}
which are submultiplicative for $s\geq0$.\par
If $\cA\in \mathrm{GL}(d,\bR )$, then  $|\cA z|$ defines an equivalent
norm on $\rdd$, hence for every $s\in\R$, there exist $C_1,C_2>0$ such
that
\begin{equation}\label{pesieq}
C_1 v_s(z)\leq v_s(\cA z)\leq C_2 v_s(z),\quad \forall z\in\rdd.
\end{equation}
 For $s\geq0$, we denote by $\mathcal{M}_{v_s}(\rdd)$ the space of $v_s$-moderate weights on $\rdd$;
these  are measurable positive functions $m$ satisfying $m(z+\zeta)\leq C
v_s(z)m(\zeta)$ for every $z,\zeta\in\rdd$.

\begin{definition}  \label{prva}
Given  $g\in\cS(\rd)$, $s\geq0$, a  weight
function $m\in\mathcal{M}_{v_s}(\rdd)$, and $1\leq p,q\leq
\infty$, the {\it
  modulation space} $M^{p,q}_m(\Ren)$ consists of all tempered
distributions $f\in \cS' (\rd) $ such that $V_gf\in L^{p,q}_m(\Renn )$
(weighted mixed-norm spaces). The norm on $M^{p,q}_m(\rd)$ is
\begin{equation}\label{defmod}
\|f\|_{M^{p,q}_m}=\|V_gf\|_{L^{p,q}_m}=\left(\int_{\Ren}
  \left(\int_{\Ren}|V_gf(x,\o)|^pm(x,\o)^p\,
    dx\right)^{q/p}d\o\right)^{1/q}  \,
\end{equation}
(with obvious changes for $p=\infty$ or $q=\infty$).
\end{definition}
 When $p=q$, we simply write $M^{p}_m(\rd)$ instead of
 $M^{p,p}_m(\rd)$. The spaces $M^{p,q}_m(\rd)$ are Banach spaces,  and
 every nonzero $g\in M^{1}_{v_s}(\rd)$ yields an equivalent norm in
 \eqref{defmod}. Thus  $M^{p,q}_m(\Ren)$ is independent of the choice
 of $g\in  M^{1}_{v_s}(\rd)$.
\par We  recall the inversion formula for
the STFT (see  (\cite[Proposition 11.3.2]{book}): assume $g\in M^{1}_v(\rd)\setminus\{0\}$,
 $f\in M^{p,q}_m(\rd)$, then
\begin{equation}\label{invformula}
f=\frac1{\|g\|_2^2}\int_{\R^{2d}} V_g f(z) \pi (z)  g\, dz \, , 
\end{equation}
and the  equality holds in $M^{p,q}_m(\rd)$.\par
 The adjoint operator of $V_g$,  defined by
 $$V_g^\ast F(t)=\intrdd F(z)  \pi (z) g dz \, , 
 $$
 maps the Banach space $L^{p,q}_m(\rdd)$ into $M^{p,q}_m(\rd)$. In particular, if $F=V_g f$ the inversion formula \eqref{invformula} becomes
 \begin{equation}\label{treduetre}
 {\rm Id}_{M^{p,q}_m}=\frac 1 {\|g\|_2^2} V_g^\ast V_g.
 \end{equation}
Finally we say that a measurable function $F(x,\eta)$ belongs to the amalgam space $W(L^\infty,L^{p,q}_m)$ if the function
\[
\tilde{F}(x,\eta):=\essupp _{(u,v)\in [0,1]^{2d}}|F(x+u,\eta+v)|
\] belongs to $L^{p,q}_{m}$; we then set
\[
\|F\|_{W(L^\infty,L^{p,q}_m)}:=\|\tilde{F}\|_{L^{p,q}_{m}}.
\]
\par
We denote by $W(C^0,L^{p,q}_m)$ the subspace of continuous functions,
with the induced norm
$\|F\|_{W(C^0,L^{p,q}_m)}=\|F\|_{W(L^\infty,L^{p,q}_m)}$.
The amalgam spaces obey the  following convolution relation
\cite[Theorem 11.1.5]{book}: if $1\leq p,q\leq\infty$, $s\geq0$, $m\in
\mathcal{M}_{v_s}(\rdd)$, then
\begin{equation}\label{conv-amalgam}
L^{p,q}_m\ast W(C^0,L^{1}_{v_s})\subset W(C^0,L^{1}_{v_s}).
\end{equation}

\subsection{Metaplectic Operators}

Given a symplectic matrix $\cA \in Sp(d,\R)$, the corresponding
metaplectic operator $\mu (\cA )$ is defined by the intertwining
relation
\begin{equation}\label{metap1}
\pi (\cA z) = c_\cA  \, \mu (\cA ) \pi (z) \mu (\cA )\inv  \quad  \forall
z\in \rdd \, ,
\end{equation}
where $c_\cA \in \bC , |c_{\cA } | =1$ is a phase factor (for details, see e.g. \cite{folland89}).
The following properties of metaplectic operators are important in our
\tf\ approach to FIOs.

\begin{lemma} \label{lkh1}
  (i) Fix $g\in \cS (\rd )$ and $\A \in Sp(d,\R)$, then, for all $N\geq
  0$,
  \begin{equation}
    \label{eq:kh1}
    |\langle \mu (\A )\pi (z)g, \pi (w)g\rangle | \leq C_N \langle
    w-\A z\rangle ^{-N}  \, .
  \end{equation}
(ii) If $\sigma \in \mif _{1\otimes v_s} (\rdd) $ and $\A \in Sp(d,\R)$,
then $\sigma \circ \A \in \mif _{1\otimes v_s}(\rdd)$ and
\begin{equation}
  \label{eq:kh2}
\|\sigma \circ \A \inv \|_{\mif _{1\otimes v_s}} \leq   \|(\A ^T)\inv
\|^s \, \|V_{\Phi
  \circ \A  } \Phi \|_{L^1_{v_s} } \|\sigma   \|_{\mif
  _{1\otimes v_s}},
\end{equation}
for a non-zero window $\Phi \in \cS (\rdd )$.
\end{lemma}

\begin{proof}
  (i) is Proposition~5.3 from~\cite{Wiener}.

(ii) The invariance is folklore. To see \eqref{eq:kh2},
fix a non-zero window $\Phi \in \cS (\rdd )$. Then
\begin{align*}
  \|\sigma \circ \A \inv  \|_{\mif _{1\otimes v_s}} &= \intrdd \sup _{z\in
    \rdd } |\langle \sigma \circ \A \inv , M_\zeta T_z \Phi \rangle
  | v_s(\zeta ) \, d\zeta \\
&= \intrdd \sup _{z\in
    \rdd } |\langle \sigma  , M_{\A ^T \zeta } T_{\A \inv
    z} (\Phi \circ \A ) \rangle
  | v_s(\zeta ) \, d\zeta \\
&= \intrdd \sup _{z\in
    \rdd } |\langle \sigma  , M_{ \zeta } T_{z
    } (\Phi \circ \A \rangle )
  | v_s((\A ^T)\inv \zeta )  \, d\zeta \, .
\end{align*}
First, since $|(\A ^T)\inv \zeta| \leq \|(\A ^T)\inv \| \, | \zeta |$
where
we use the spectral norm $\|\A\|$ of a matrix, we obtain $v_s((\A
^T)\inv \zeta ) \leq   \|(\A ^T)\inv
\|^s v_s(\zeta )$. Second, the
change of windows yields an equivalent norm of $\mif _{1\otimes v_s}$,
and the estimate (11.31) of \cite{book} yields the following estimate:
$$
  \|V_{\Phi \circ \A }\sigma \|_{L^{\infty , 1} _{1\otimes v_s}} \leq
  \|V_{\Phi \circ \A  } \Phi \|_{L^1_{v_s} } \|V_\Phi \sigma
  \|_{L^{\infty , 1} _{1\otimes v_s}} \, .
$$
The combination of  both estimates yields
$$
\|\sigma \circ \A \inv \|_{\mif _{1\otimes v_s}} \leq   \|(\A ^T)\inv
\|^s\, \|V_{\Phi
  \circ \A  } \Phi \|_{L^1_{v_s} } \|\sigma   \|_{\mif
  _{1\otimes v_s}} \, .
$$
\end{proof}

\subsection{Pseudodifferential Operators}
In the proofs of the main theorems we will make decisive use of the
following properties of pseudodifferential operators  with a symbol in
a generalized Sj\"ostrand class. To shorten the notations, from now on we will use the $\sigma^w$ notation instead of $\sigma^w(x,D)$.

\begin{proposition}[\cite{charly06}]
  \label{charpsdo}
Fix $g \in M^1_{v_s}(\rd )$ and let $\sigma \in M^\infty _{1\otimes v_s}(\rdd )$. Then $\sigma \in
\mifs (\rdd )$ \fif\ there exists a function $H\in  L^1_{v_s}(\rdd )$ such that
\begin{equation}
  \label{eq:kh9}
  |\langle \sigma ^w \pi (w) g , \pi (z) g\rangle | \leq H(w-z) \qquad
  \forall w,z \in \rdd \, .
\end{equation}
The function $H$ can be chosen as
\begin{equation}
  \label{eq:kh12}
H(z)=\sup_{u\in\rdd}|V_{\Phi} \sigma (u,j(z))|,
\end{equation}
where $j(z)=(z_2,-z_1)$ for $z=(z_1,z_2)\in\rdd$ and the  window
function $\Phi=W(g,g)$ is  the Wigner distribution of $g$.
\end{proposition}

  We call $H$ the \emph{controlling function} associated to the symbol $\sigma$. Observe that
\begin{equation*} 
\|H\|_{L^1_{v_s}}=\int_{\rdd} \sup_{u\in\rdd}|V_{\Phi} \sigma (u,j(z))|v_s(z)dz = \int_{\rdd} \sup_{u\in\rdd}|V_{\Phi} \sigma (u,z)|v_s(z)dz=\|\sigma\|_{M^{\infty,1}_{1\otimes v_s}}.
\end{equation*}

As a consequence of this characterization one obtains the following
properties of \psdo s. See \cite{wiener31} for the unweighted case and
\cite{charly06} for the general case and an  approach in the spirit
of \tfa .

\begin{proposition}\label{fund}
Assume that $\sigma \in \mifs (\rdd)$. Then:

(i) \emph{Boundedness:} $\sigma ^w $ is bounded on every \modsp\ $M^{p,q}_m(\rd )$ for
$1\leq p,q \leq \infty   $ and every $v_s$-moderate weight $m$.

(ii) \emph{Algebra property:} If $\sigma _1 , \sigma _2 \in \mifs $,
then  $\sigma _1^w \sigma _2^w = \tau ^w$ with a symbol $\tau \in
\mifs  (\rdd)$.

(iii) \emph{Wiener property:} If $\sigma ^w $ is invertible on $\lrd
$, then $(\sigma ^w)\inv = \tau ^w$  with a symbol $\tau \in
\mifs  (\rdd)$.
\end{proposition}

 In the context of Definition~\ref{def1.1} the class of pseudodifferential
operators with a symbol in $\mifs (\rdd )$ is just $FIO (\mathrm{Id},
v_s)$. Our results are therefore a generalization of
Proposition~\ref{charpsdo}, but the proofs are based on this special
case.

\section{Properties of the class $FIO(\A,v_s)$}

We first show that Definition~\ref{def1.1} of generalized metaplectic
operators is independent of the auxiliary window $g$.  In the
following, we will use the continuous definition of $FIO(\cA , v_s)$,
whereas in~\cite{Wiener} we have used a  discrete definition of
a related class of FIOs.

\begin{proposition}\label{prop3.1}
  The definition of the class $FIO(\cA ,v_s)$ is independent of the
  window function $g\in\cS(\rd)$. Moreover, if $T \in FIO (\cA ,
  v_s)$, then there exists a  function $H\in
  W(C^0 ,L^1_{v_s})$ so that the decay condition~\eqref{asterisco}
  holds.
\end{proposition}
\begin{proof}
We verify both  statements simultaneously. Assume that
\eqref{asterisco} holds  for some function $H\in L^1_{v_s}(\rdd)$.
To switch between  a given non-zero  window function  $g\in \cS (\rd
)$ and   some other non-zero  $\f\in\cS(\rd)$ we use the inversion
formula \eqref{treduetre} and write
 $\pi(z)\f=\frac1{\|g\|^2_2}\intrdd \la\pi(z)\f,\pi(p)g\ra
 \pi(p)g\,d p$ with convergence in $\cS(\rd)$ for every  $z \in \rdd$. Consequently
$T\pi(z)\f =\frac1{\|g\|^2_2}\intrdd \la\pi(z)\f,\pi(p)g\ra
T\pi(p)g dp$  converges weak$^*$ in $\cS'(\rd)$
and the following identity  is well-defined:
$$
\langle T\pi(z)\f,\pi(w)\f\ra=\frac1{\|g\|^4_2}\intrdd \intrdd
\la\pi(z)\f,\pi(p)g\ra \la T\pi(p)g,\pi(q)g\ra \overline{ \la
  \pi(w)\f,\pi(q)g\ra } \,dp dq \, .
$$
Hence, using  \eqref{asterisco},
\begin{align*}
|\langle T\pi(z)\f,\pi(w)\f\ra|&\leq\frac1{\|g\|^4_2}\intrdd
\intrdd|\la\pi(z)\f,\pi(p)g\ra|\,| \la T\pi(p)g,\pi(q)g\ra|\, \\
&\qquad\qquad\qquad\qquad\qquad\qquad\qquad\qquad\times |\la\pi(q)g,
\pi(w)\f\ra| \,dp dq\\
&\leq \frac1{\|g\|^4_2}\intrdd \intrdd H(q-\cA p) |V_g\f(p-z)| \,|V_\f g(w-q)|\,d p dq\\
&\leq \frac1{\|g\|^4_2}\intrdd \left(\intrdd H(q-\cA p) |V_g\f(p-z) |
  \, dp\right) \,|V_\f g(w-q)| dq.
\end{align*}
First, the change of variables $\cA p=x$ (so that $dp=dx$ for $\cA\in
\Spnr $), and second,  the effect of the linear
transformation $\cA$ on the short-time Fourier transform contained in
\cite[Lemma 9.4.3]{book} allow the following expression of the inner
integral above

\begin{align*}\intrdd H(q-\cA p) |V_g\f(p-z)| dp&=\intrdd H(q-x) |V_g\f(\cA^{-1}x-z)| dx\\ & =\intrdd H(q-x) |V_g\f(\cA^{-1}(x-\cA z))| dx\\
& = \intrdd H(q-x) |V_{\mu(\cA) g} (\mu(\cA)\f)(x-\cA z)| dx\\
&=(H\ast|V_{\mu(\cA) g} \mu(\cA)\f|)(q-\cA z)
\end{align*}

Since  $\cS $ is invariant under metaplectic operators and $\f , g
\in \cS (\rd )$, we have  $\mu(\cA)\f , \mu(\cA)g \in \cS (\rd )$ and
hence $V_{\mu(\cA) g}( \mu(\cA)\f ) \in\cS (\rdd )$ (see,
e.g. \cite{GZ}).  Now the convolution relation~\eqref{conv-amalgam}
implies that
$$G:=H\ast|V_{\mu(\cA) g} \mu(\cA)\f|\in L^1_{v_s}(\rdd)\ast\cS(\rdd)\subset W(C^0,L^1_{v_s})$$
for every $s\geq 0$. The exterior integral is then
$$\intrdd G(q-\cA z) |V_\f g(w-q)|\,dq=G\ast|V_\f g|(w-\cA z):=F(w-\cA z),
$$
 and  $F\in W(C^0,L^1_{v_s})$, as claimed.
 \end{proof}

\begin{lemma}\label{normeeq}
Fix  $s\in\R$  and  $M\in GL(2d,\R)$. Then $L^1_{v_s}(\rdd )$ is
invariant under $M$, i.e., if $H\in L^1_{v_s}(\rdd)$,  then  $H\circ M\in  L^1_{v_s}(\rdd)$.
\end{lemma}
\begin{proof}
Since $v_s(M\inv x) \asymp v_s(x)$ by \eqref{pesieq}  for every
$M\in GL(2d,\R)$,  we obtain
$$
\|H\circ M \|_{L^1_{v_s}} \asymp \|H\|_{L^1_{v_s \circ M\inv }} \asymp
\|H \|_{L^1_{v_s}} \, .
$$
\end{proof}

In the following we derive the boundedness, composition and
invertibility properties of generalized metaplectic operators in the
style of Proposition~\ref{fund}. The results  are analogous
to the results for $M^\infty _{1\otimes v_s} (\rdd )$ in~\cite{Wiener}, though the
proofs are a bit more technical.

\begin{theorem}\label{teo3.3}
Fix  $\A\in\Spnr$, $s\geq0$, $m\in\cM_{v_s}$ and let $T$ be
generalized metaplectic operator in $ FIO(\A,v_s)$. Then $T$ is
bounded from $M^p_m(\rd)$ to $M^p_{m\circ\A^{-1}}(\rd)$, $1\leq p\leq
\infty$. In particular, $T$ is bounded on $M^p_{v_s}(\rd )$ and
$M^p_{1/v_s}(\rd )$.
\end{theorem}
\begin{proof}
We fix  $g\in\cS(\rd)$ with the normalization  $\|g\|_2=1$, so that
the inversion formula \eqref{treduetre} is simply $V_g^\ast V_g={\rm
  Id}$.  Writing  $T$ as
$$T=V_g^\ast V_g T V_g^\ast V_g,$$
then $V_g T V_g^\ast$ is an integral operator with kernel
$$K_T(w,z)=\la T\pi(z)g,\pi(w)g\ra,\quad w,z\in\rdd \, .
$$
($K_T$  is sometimes  called the essential kernel of the operator.)
By definition, $V_g$ is bounded from $M^p_m(\rd )$ to
$L^p_m(\rdd )$,  and $V_g^*$ is bounded from $L^p_m(\rdd )$ to
$M^p_m(\rd )$.
Consequently, if $V_g T V_g^\ast$ is bounded from $L^p_m(\rdd)$ to
$L^p_{m\circ\A^{-1}}(\rdd)$, then  $T$ is bounded from $M^p_m(\rd)$ to
$M^p_{m\circ\A^{-1}}(\rd)$.
Using the defining property  \eqref{asterisco} of $FIO (\A, v_s)$,  we
write, for  $F\in L^p_m(\rdd)$,
\begin{align*}
|V_g T V_g^\ast F(w)|& = \Big|\intrdd K_T(w,z) F(z) \, dz \Big| \leq \intrdd
|F(z)| H(w-\A z) dz \\
& =\intrdd |F(z)| (H\circ \A)(\A^{-1}w-z) dz\\
&=F\ast(H\circ\A)(\A^{-1}w).
\end{align*}
By Lemma \ref{normeeq},   $H\circ\A$ is again  in $L^1_{v_s}(\rdd)$,  so
that  $F\ast(H\circ\A)\in L^p_m(\rdd)\ast L^1_{v_s}(\rdd)\subset
L^p_m(\rdd)$. This shows that   $V_g T V_g^\ast F\in
L^p_{m\circ\A\inv}(\rdd)$, as desired.  Since $v_s\circ\A\asymp v_s$, $T$
leaves $M^p_{v_s}$ invariant.
\end{proof}
\begin{theorem}\label{prod1}
Let  $\A_i\in\Spnr$, $s_i\in\R$, $s=\min\{s_1,s_2\}$, and  $T_i\in FIO(\A_i,v_{s_i})$, $i=1,2$. Then $T_1T_2\in  FIO(\A_1\A_2,v_s)$.
\end{theorem}
\begin{proof}
Again fix  $g\in\cS(\rd)$ with $\|g\|_2=1$, so that $V_g^\ast V_g={\rm Id}$. Then write
$$T_1T_2=V_g^\ast V_g T_1T_2V_g^\ast V_g=V_g^\ast(V_g T_1 V_g^\ast) (V_gT_2V_g^\ast) V_g.
$$
Let  $H_i \in L^1_{v_s}(\rdd )$ be  the  function  controlling  the
kernel of  $T_i$ defined  in \eqref{asterisco}, $i=1,2$. Then an easy computation shows
\begin{align}
|\la T_1T_2\pi(z)g,\pi(w)g\ra|&\leq\intrdd |\la
T_2\pi(z)g,\pi(u)g\ra| \, |\la T_1\pi(u)g,\pi(w)g\ra| du \notag \\
&\leq\intrdd H_1(w-\A_1u)H_2(u-\A_2 z)du \notag\\
&=\intrdd (H_1\circ\A_1)(A_1^{-1}w-u)H_2(u-\A_2 z)du \notag \\
&=((H_1\circ\A_1)\ast H_2)\circ\A_1^{-1})(w-\A_1\A_2 z). \label{eq:kh4}
\end{align}
 Using Lemma \ref{normeeq} and
the convolution relation $L^1_{v_{s_1}}\ast L^1_{v_{s_2}}\subset L^1_{v_s}$,
we conclude that $T_1T_2 \in FIO (\cA _1 \cA _2, v_s)$.
\end{proof}

For later use we isolate a fact about the essential kernel of a \psdo .
\begin{cor}
Fix a non-zero $g\in \cS (\rd )$ and   let $\sigma _j \in \mifs (\rdd) $, $j=1, \dots ,n $ and $H_j \in
  L^1_{v_s}(\rdd )$ be controlling function  associated to $\sigma
  _j $. Then
  \begin{equation}
    \label{eq:kh3}
    |\langle \sigma _1^w \sigma _2 ^w \dots \sigma _n^w \pi (z)g, \pi
    (w)g\rangle | \leq (H_1 \ast H_2 \ast \dots H_n ) (w-z) \, \quad
    w,z\in \rdd \, .
  \end{equation}
\end{cor}

\begin{proof}
 The characterization of \psdo s (Proposition~\ref{charpsdo})  says that $\sigma _j \in \mifs (\rdd )$, if and only if
$   |\langle \sigma _j^w \pi (z)g, \pi
    (w)g\rangle | \leq H_j  (w-z) $ for some $H_j \in L^1_{v_s}(\rdd
    )$. The statement now follows from  ~\eqref{eq:kh4} with $\A _j =
    \mathrm{Id}$ and induction.
\end{proof}

Recall from the introduction that
$$FIO(\Spnr,v_s)=\bigcup_{\A\in\Spnr}FIO(\A,v_s) \, .
$$
\begin{corollary}\label{cor3.5}
For every $s\geq0$, the class of generalized metaplectic operators  $FIO(\Spnr,v_s)$ is an algebra with respect to the composition of operators.
\end{corollary}
Finally, we study the invertibility in  $FIO(\A,v_s)$.
\begin{theorem} \label{Wineroper} Let $T\in FIO(\A,v_s)$ with
  $s\geq0$. If  $T$ is invertible on $L^2(\rd)$,
then $T^{-1} \in FIO(\A^{-1},v_s)$. Consequently, the algebra $FIO(\Spnr,v_s)$ is inverse-closed in $B (\lrd )$.
\end{theorem}
\begin{proof}  The arguments are  similar  to \cite[Theorem
  3.7]{Wiener}. First we show that  the adjoint  operator $T^\ast$
  belongs to the class $FIO(\A \inv, v_s)$. Indeed,
\begin{align*}|\langle T^*\pi(z)g,\pi(w) g \rangle|&=|\langle
  \pi(z)g,T(\pi(w) g) \rangle| \leq H(z-\A w)\\ &=(H\circ\A)(\A^{-1}z-w)
=(H\circ\A)(-(w-\A^{-1}z)).
\end{align*}
By Lemma \ref{normeeq} the function $z \mapsto H\circ\A(-z)$ is in $L^1_{v_s}(\rdd)$. Hence, by Theorem \ref{prod1}, the operator $P:=T^\ast T$ is in
$FIO(\mathrm{Id},v_s)$ and thus satisfies the estimate $|\langle P\pi
(z )g, \pi (w )g\rangle | \lesssim H(w-z)$  for some  $H\in L^1_{v_s}(\rdd)$.
 By Proposition~\ref{charpsdo} 
$P$ is a pseudodifferential operator with a  symbol in
$M^{\infty,1}_{1\otimes v_s}(\rdd)$. Since $T$ and therefore
$T^\ast$ are invertible on $L^2(\rd)$,  $P$ is also  invertible on
$L^2(\rd)$.  Since the class of pseudodifferential operators with a
 symbol in $M^{\infty,1}_{1\otimes v_s}(\rdd)$ is inverse-closed in $B (\lrd )$ (\cite{GR,wiener31}),
  the inverse $P^{-1}$ is again a  pseudodifferential
operator with a symbol in $M^{\infty,1}_{1\otimes v_s}(\rdd)$.
Hence $P^{-1}$ is in $FIO(\mathrm{Id},v_s)$. Finally, using the
algebra property  of Theorem \ref{prod1} once more, we obtain that $T^{-1}=P^{-1}
T^\ast$ is in $FIO(\mathrm{Id}\circ\A \inv ,v_s)=FIO(\A \inv ,v_s)$, as claimed.
\end{proof}
\begin{theorem}\label{pseudomu}
Fix $s\geq0$ and $\cA \in \Spnr $.  A linear continuous operator $T:
\cS(\rd)\to\cS'(\rd)$ is in   $FIO(\A,v_s)$ if and only if there
exist symbols $\sigma_1, \sigma_2 \in
M^{\infty,1}_{1\otimes v_s}(\rdd)$, 
such that
\begin{equation}\label{pseudomu1}
T=\sigma_1^w\mu(\A)\quad \mbox{and}\quad
T=\mu(\A)\sigma^w_2.
\end{equation}
The symbols $\sigma _1$ and $\sigma _2$ are related by
\begin{equation}\label{hormander}\sigma_2=\sigma_1\circ\A.\end{equation}
\end{theorem}
\begin{proof}
First, assume  $T\in FIO(\A,v_s)$. We prove the factorization
$T=\sigma_1^w\mu(\A)$.
Since  $\mu(\A)^{-1}=\mu(\A^{-1})$ is in $FIO( \cA \inv ,v_s)$ by
Lemma~\ref{lkh1}(i),   the algebra property of Theorem~\ref{prod1}
implies that
$T\mu(\A^{-1})\in FIO(\mathrm{Id},v_s)$. Now  Proposition~\ref{charpsdo}
 implies the existence of a symbol
$\sigma _1 \in M^{\infty,1}_{1\otimes v_s}(\rdd)$, such that
$T\mu(\A)^{-1}=\sigma^w_1$, as  claimed.\par

Since $T=\mu (\A ) \mu (A)\inv
\sigma_1^w\mu(\A) = \mu (\A ) (\sigma _1 \circ \A ) ^w $ by the
symplectic invariance properties of the Weyl calculus
(e.g. \cite[Theorem 18.5.9]{hormander3}), the
factorization $T=\mu (\cA ) \sigma _2^w$  and \eqref{hormander} follow immediately.

Conversely, assume $T=\sigma^w_1\mu(\A)$ with symbol $\sigma\in
M^{\infty,1}_{1\otimes v_s}$. Since
$\mu(\A)\in FIO(\A,v_s)$ by \eqref{eq:kh1} and  $\sigma^w_1\in
FIO(\mathrm{Id},v_s)$ by Proposition~\ref{charpsdo},   Theorem \ref{prod1} yields that
$T=\sigma^w_1\mu(\A)\in FIO(\A\circ
\mathrm{Id},v_s)=FIO(\A,v_s)$, as desired.
\end{proof}
\begin{remark}\rm  The characterization \eqref{pseudomu1} carries over
  to other calculi of pseudodifferential operators, in particular the
  Kohn-Nirenberg correspondence, whereas \eqref{hormander} is peculiar
  of the Weyl correspondence. 
\end{remark}

The characterization of Theorem~\ref{pseudomu} offers an alternative
approach to the definition and analysis of generalized metaplectic
operators. We could have defined  the class $ FIO(\A,v_s)$ by the
factorization \eqref{pseudomu1}. Then the basic properties of
generalized metaplectic operators can be derived directly from this
definition and Proposition~\ref{fund}.  The mapping properties of generalized
metaplectic operators  follow from the known mapping properties of
$\mu (\A )$ and of pseudodifferential operators. If $S= \sigma ^w
\mu (\A _1)$ and $T=\tau ^w \mu (\A _2)$, then 
$$
ST =  \sigma ^w\mu (\A _1) \tau ^w \mu (\A _2) = \sigma ^w
 \big( \tau \circ \A _1)^w \mu (\A _1)  \mu (\A _2)  \, .
$$ 
Thus $ST \in FIO (\A _1\A _2, v_s)$ by Proposition~\ref{fund}(ii).
 If $T = \sigma ^w  \mu (\cA )  \in FIO(\A , v_s) $ is
invertible on  $\lrd $, then so it $\sigma ^w = T\mu (\A \inv
)$. Once more by Proposition~\ref{fund}(iii) the inverse $(T\mu (\A \inv
))\inv = \tau ^w$ possesses a  symbol $\tau $ in $\mifs
$. Consequently $T\inv = \mu (\A \inv) \tau ^w$ and
by~\eqref{pseudomu1} $T\inv \in FIO(\A \inv , v_s)$. 

With this approach, however, the proof of the 
off-diagonal decay~\eqref{def1.1} of $FIO(\A , v_s)$ requires more
work.

\section{Application to Schr\"odinger Equations}
As an application, we  study the Cauchy problem
  for linear Schr\"odinger
  equations of the type
  \begin{equation}\label{C1}
\begin{cases} i \displaystyle\frac{\partial
u}{\partial t} +H u=0\\
u(0,x)=u_0(x),
\end{cases}
\end{equation}
with $t\in \bR$ and the initial condition $u_0\in\cS(\rd)$ or some
larger space.  We consider  a Hamiltonian of the form
\begin{equation}\label{C1bis}
H=a^w+ \sigma^w,
\end{equation}
 where $a^w$ is the Weyl quantization of a real homogeneous  quadratic polynomial  on
$\rdd$ and  $\sigma^w$ is a pseudodifferential operator  with a
symbol $\sigma\in M^{\infty,1}_{1\otimes v_s}$, $s\geq0$. If $\sigma
\in \mif _{1\otimes v_s} (\rdd )  \subseteq \mif (\rdd )  $, then $\sigma ^w$ is bounded
on $\lrd $ and on all \modsp s $M^p _m$ for $1\leq p\leq \infty$ and
$v_s$-moderate weight $m$.  Consequently,  $H = a^w + \sigma ^w$ is a
bounded perturbation of the generator $H_0 = a^w$  of a unitary
group by~\cite{RS75}, and $H$ is the generator of a well-defined
(semi-)group.

This set-up includes the usual Schr\"odinger equation $H= \Delta -
V(x)$  and the perturbation of the
harmonic oscillator $H= \Delta - |x|^2 - V(x)$ with a potential $V\in
M^{\infty,1} (\rd )$. Note that $V$ is bounded, but not necessarily smooth.

We first consider the unperturbed  case $\sigma=0$, namely
\begin{equation}\label{C12}
\begin{cases} i \displaystyle\frac{\partial
u}{\partial t} +a^w u=0\\
u(0,x)=u_0(x).
\end{cases}
\end{equation}
Then  the solution is given by metaplectic operators
$u=\mu(\mathcal{A}_t)u_0$, for a suitable symplectic matrix
$\mathcal{A}_t$. Precisely, if $a(x,\xi)=\frac{1}{2}xAx+\xi B
x+\frac{1}{2}\xi C\xi$, with $A, C$ symmetric and $B$ invertible, we
can consider the classical evolution, given by the linear Hamiltonian
system
\[
\begin{cases}
2\pi \dot x=\nabla  _\xi a  =Bx+C\xi\\
2\pi \dot \xi=-\nabla _x a =-A x-B^T\xi
\end{cases}
\]
(the factor $2\pi$ depends on our normalization of the Fourier transform) with Hamiltonian matrix $\mathbb{A}:=\begin{pmatrix}B&C\\
  -A&-B^T\end{pmatrix}\in \mathrm{sp} (d,\bR )$. Then we have
$\mathcal{A}_t=e^{t\mathbb{A}}\in Sp(d,\R)$, and $\mu(\mathcal{A}_t)$
is given by the formula \eqref{eq:kh21} with $\mathcal{A}$ replaced by
$\mathcal{A}_t$ with $\sigma \equiv 1$  (see \cite{folland89} and \cite[Section 5]{cn} for
details). \par
Our main result shows that the propagator of the   perturbed problem
\eqref{C1} -- \eqref{C1bis}  
is  a generalized metaplectic operator and thus can be written as a
Fourier integral operator with a quadratic phase.
\begin{theorem}\label{teofinal}
Let $H = a^w + \sigma^w$ be a Hamiltonian with a real quadratic
homogeneous polynomial $a$ and  $\sigma\in
M^{\infty,1}_{1\otimes v_s}(\rdd ) $ and $U (t) = e^{itH}$ the corresponding
propagator. Then $U(t)$ is a generalized metaplectic operator for each
$t >0$.
Specifically, the solution of the homogenous problem $iu_t + a^w u = 0$ is given by a
metaplectic operator $\mu(
\mathcal{A}_t)$, and $e^{itH}$
is of the from
\[
e^{itH} = \mu(\mathcal{A}_t)b^w_t
\]
for some symbol $b_t\in M^{\infty,1}_{1\otimes v_s}(\rdd)$.
\end{theorem}
\begin{corollary}\label{corfinal}
 If the initial condition $u_0$ is in a modulation space $M^p_
m(\rd)$, then $u(t,\cdot)\in M^p_{m\circ \mathcal{A}^{-1}_t}(\rd)$ for every $t>0$.
In particular, if $m\asymp m\circ \mathcal{A}$
for all $\mathcal{A} \in Sp(d,\R)$ (as is the case for $v_s$), then
the time evolution leaves $M^p_
m(\rd)$ invariant. In other words, the Schr\"odinger evolution
preserves the phase space concentration of the initial condition.
\end{corollary}
\begin{proof}[Proof of Corollary \ref{corfinal}]
It follows from Theorem \ref{teofinal} and the representation in
Theorem \ref{pseudomu} that $e^{itH}\in FIO(\A,v_s)$, so that the
claim follows from Theorem \ref{teo3.3}.
\end{proof}

Before we turn to the proof of Theorem \ref{teofinal}, we recapitulate the facts about the
bounded perturbation of operator (semi)groups. We refer to the
textbooks~\cite{RS75} and \cite{EN06} for an introduction to operator
semigroups.

Let $H_0$ be the generator of a strongly continuous one-parameter group  on a Banach space
$\cX $ and $T(t) = e^{it H_0} $ be the corresponding (semi)group that
solves the evolution equation $i \frac{dT(t)}{dt} = H_0 T(t)$. Let $B$
be a bounded operator on $\cX $ and $H=H_0 +B$. Then $H$ is the
generator of a strongly continuous one-parameter group $S(t)$~\cite{EN06}.

In our case
$H_0 = a^w$
for a real homogeneous quadratic polynomial and $\sigma \in \mifs
$. As explained above,  $T(t) = e^{i t H_0} = \mu (\cA _t)$ for a one-parameter group
$t \mapsto \cA _t \in Sp(d,\R) $ of symplectic matrices.
The Banach space is a \modsp\  $M^p_m(\rd )$ with a weight satisfying
$m \asymp m\circ \cA $ for every $\cA \in Sp(d,\R )$. Then $e^{itH_0}
= \mu (\cA _t)$ is a strongly continuous one-parameter group on
$M^p_m(\rd )$. The perturbation is a \psdo\ $B=\sigma ^w$ with a
symbol $\sigma \in \mifs (\rdd)$. Since by
Proposition~\ref{fund}(i) $\sigma ^w$ is bounded on $M^p_m(\rd)$, whenever
$m$ is  $v_s$-moderate and $1\leq p\leq \infty $, $B$ is a
bounded perturbation of $H_0$. Therefore the perturbed Hamiltonian $H=
a^w + \sigma ^w$ generates a strongly continuous semigroup
on $M^p_m(\rd )$. Thus the second statement of Corollary~\ref{corfinal}
also follows directly from the standard theory~\cite[Ch.~3, Cor.~1.7]{EN06}.

For the stronger statement of Theorem~\ref{teofinal} we need to work
harder and explain how the perturbed semigroup is constructed. The
perturbed semigroup $S(t) = e^{itH}$ satisfies an  abstract Volterra
equation
\begin{equation}
  \label{eq:kh6}
  S(t)x = T(t)x + \int _0^t T(t-s) B S(s) x \, ds =
  T(t)\Big(\mathrm{Id} + \int _0^t T(-s) B T(s) T(-s) S(s)  \, ds
  \Big) x
\end{equation}
for every $x\in \cX $ and $t\geq 0$. The correction operator $C(t) =
T(-t)S(t)$ therefore satisfies the Volterra equation
\begin{equation}
  \label{eq:kh7}
  C(t) =  \mathrm{Id} + \int _0^t T(-s) B T(s) C(s)   \, ds  \, ,
\end{equation}
where the integral is to be understood in the strong sense.
Now write $B(s) = T(-s) B T(s)$, then the solution of~\eqref{eq:kh7}
can be written as a so-called \emph{Dyson-Phillips expansion}
(\cite[X.69]{RS75} or \cite[Ch.~3, Thm.~1.10]{EN06})
\begin{equation}
  \label{eq:kh8}
  C(t) = \mathrm{Id} + \sum _{n=1}^\infty (-i)^n \int _0^t \int
  _0^{t_1} \dots \int
  _0^{t_{n-1}} B(t_1) B(t_2) \dots B(t_n) \, dt_1 \dots dt_n := \sum
  _{n=0}^\infty C_n(t) \, .
\end{equation}


\begin{proof}[Proof of Theorem \ref{teofinal}]
In our case $T(t) = \mu (\A _t)$,  $B = \sigma ^w $, and
$$
B(t) = \mu (\A _{-t}) \sigma ^w \mu (\A _t) = (\sigma \circ \A _t
\inv )^w  \, .
$$
We want to show that $C(t)$ in~\eqref{eq:kh8}
is a \psdo\ with a (Weyl) symbol in $\mifs (\rdd)$.  Since in general the
mapping $t\mapsto (\sigma \circ \A _t \inv )^w$ is not continuous
with respect to the $\mifs $-norm,  we have to exert special care with
the vector-valued integrals in~\eqref{eq:kh8}. We will use  the
characterization  of pseudodifferential operators
(Proposition~\ref{charpsdo}) and show that $|\langle C_n(t) \pi (z)g,
\pi (w)g\rangle | \leq \cH _n(t)$ for a controlling function in
$L^1_{v_s}(\rdd )$ for every $t>0$, such that $\sum_n \cH _n(t) \in
L^1_{v_s}(\rdd )$. This procedure  will involve only a weak
interpretation of these integrals.
We proceed in several steps.

\emph{Step 1: Pointwise estimate of  the integrand of~\eqref{eq:kh8}.}
Let $H_t \in L^1_{v_s}(\rdd )$ be the controlling function of $\sigma
\circ \A _t\inv $. Then by~\eqref{eq:kh3} we have
$$
|\langle \prod _{j=1}^n B(t_j) \pi (z) g, \pi (w) g\rangle | \leq
( H_{t_1} \ast H_{t_2} \ast \dots \ast H_{t_n} ) (w-z) \, .
$$
According to \eqref{eq:kh8} the natural controlling function for
$\langle C_n(t) \pi (z)g, \pi (w)g\rangle $  will be the function
$$
\cH _n(t) = \int _0^t \int
  _0^{t_1} \dots \int
  _0^{t_{n-1}} H_{t_1} \ast H_{t_2} \ast \dots \ast H_{t_n}  \, dt_1
  \dots dt_n \, ,
$$
if $\cH _n (t)$ exists.

\emph{Step 2: Measurability of $H_t$.} We note that the function
$(t,u,z) \mapsto \langle \sigma \circ \A _t\inv , M_{j(z)} T_u \Phi
\rangle =  \langle \sigma  , M_{\A ^T_t j(z)} T_{\A _t \inv u} (\Phi
\circ \A _t)\rangle $  is continuous for every $\sigma \in \cS ' (\rdd
)$ and $\Phi \in \cS (\rdd )$, because $t\mapsto \A _t$ is continuous
and \tfs s are strongly continuous on $\cS (\rdd )
$~\cite{folland89,book}. Consequently the controlling function
$$
H_t(z) = \sup _{u\in \rdd } |\langle \sigma \circ \A _t\inv , M_{j(z)} T_u \Phi
\rangle|
$$
 is lower semi-continuous and thus (Borel) measurable in the variables
 $t$ and $z$.

\emph{Step 3: Boundedness of $\|H_t\|_{L^1_{v_s}}$.}
We  use Lemma~\ref{lkh1} and estimate
\begin{align*}
  \sup _{0\leq r \leq t} \|H_r\|_{L^1_{v_s}} & \leq  \sup_{0\leq r \leq
    t} \|\sigma \circ \cA _r \inv \|_{L^1_{v_s}} \\
&\leq  \|\sigma \|_{\mifs } \sup _{0\leq r \leq t} \|\A _r \|^s \,  \| V_{\Phi \circ
  \A _r} \Phi \|_{L^1_{v_s}} \leq M(t) \|\sigma \|_{\mifs } \, .
\end{align*}
For  $\Phi \in \cS (\rdd )$ it is easy to see that $M(t) = \sup _{0\leq r \leq t} \|\A _r \|^s \,  \| V_{\Phi \circ
  \A _r }\Phi \|_{L^1_{v_s}}$ is finite.

\emph{Step 4: Integrability of $\cH _n(t)$.} Since the product  $\prod _{j=1}^n H_{t_j}(z_j)$ is measurable and
non-negative, we may choose an arbitrary order of integration and then
apply Fubini-Tonelli to verify the integrability  of each term
in~\eqref{eq:kh8}. Using $$\|H_{t_1}\ast \ldots \ast
H_{t_{n}} \|_{L^1_{v_s}}\leq \prod_{j=1}^n\|H_{t_j}\|_{L^1_{v_s}},$$
 \begin{align*}
\|\cH _n (t) \|_{L^1_{v_s}}&= \intrdd
\int_0^t\int_0^{t_1}\cdots\int_0^{t_{n-1}} (H_{t_1}\ast \ldots \ast
H_{t_{n}} ) (w)  dt_1 \ldots dt_n v_s(w) dw\\
&\leq \int_0^t\int_0^{t_1}\cdots\int_0^{t_{n-1}}  \|H_{t_1}
\|_{L^1_{v_s}} \dots \|H_{t_n}
\|_{L^1_{v_s}}  \, dt_1 \ldots dt_n \\
&\leq \frac{t^n }{n!}  M(t)^n \|\sigma \|_{\mifs } ^n\, .
 \end{align*}
In the last inequality we have used the boundedness established in
Step~3.

\emph{Step 5: Combination of estimates.}
By combining all estimates we arrive at
\begin{align*}
|\langle C(t) \pi (z) g , \pi (w)g\rangle | & \leq   \sum
_{n=0}^\infty |\langle C_n(t) \pi (z) g , \pi (w)g\rangle | \\
 & \leq \sum _{n=1}^\infty \cH _n(t) (z-w) \, .
\end{align*}
By Step~4 the sum of the controlling functions is still in $L^1_{v_s}(\rdd)$
and satisfies
$$
 \|\sum _{n=1}^\infty \cH _n(t) \|_{L^1_{v_s}} \leq \sum _{n=0}^\infty
 \frac{t^n }{n!}  M(t)^n \|\sigma \|_{\mifs } ^n  = e^{tM(t)\|\sigma
   \|_{\mifs } } \, .
$$
By Proposition~\ref{charpsdo},  $C(t) = b_t ^w$ for a symbol $b_t$ in
$\mifs (\rdd)$. Finally, the propagator is
$$
e^{itH} = T(t) C(t) = \mu (\A _t) b_t^w
$$
and thus $e^{itH} $ is a generalized metaplectic operator by
Theorem~\ref{pseudomu}. This finishes the proof of Theorem~\ref{teofinal}.
\end{proof}
\begin{remark}\rm The results in Theorem \ref{teofinal} and Corollary
  \ref{corfinal} also hold  for $t<0$. Just  replace $a$, $\sigma$ by $-a$, $-\sigma$.
\end{remark}
\begin{remark}\rm In the proof of Theorem~\ref{teofinal} we have only
  used the invariance of $\mifs $ under metaplectic operators
  (Lemma~\ref{lkh1})  and  properties of the symbol class $\mif
  _{1\otimes
    v_s}(\rdd )$ stated in Proposition~\ref{fund}, namely the
  boundedness on \modsp s and  the algebra   property.

Therefore a modified version of Theorem~\ref{teofinal} holds for every
symbol class that is closed with respect to composition of operators,
invariant under metaplectic operators and bounded on \modsp s.
We mention two classes for which these assumptions are satisfied. \\

(a) Theorem~\ref{teofinal} holds for the symbol class
$M^{\infty}_{1\otimes v_s}(\rdd)$ with  $s>2d$. The relevant
properties  were proved in \cite{Wiener}.

(b) Theorem~\ref{teofinal} holds for those ``\psdo s'' whose kernel
satisfies the conditions of Schur's test. Precisely, we define $\Sigma
(\rdd )$ to contain all symbols $\sigma \in \cS ' (\rdd )$ such that
\begin{equation}
  \label{eq:c1}
  \sup _{z\in \rdd } \intrdd |\langle \sigma ^w \pi (z)g, \pi (w)
  g\rangle | \, dw < \infty \quad \text{ and } \quad    \sup _{w\in \rdd } \intrdd |\langle \sigma ^w \pi (z)g, \pi (w)
  g\rangle | \, dz < \infty \, .
\end{equation}
The same proofs as in~\cite{charly06} show that $\Sigma (\rdd )$ is  closed with respect to composition of operators,
invariant under metaplectic operators and bounded on the \modsp s
$M^p(\rd )$.
Consequently we can use $\Sigma (\rdd )$ instead of $\mif _{1\otimes
  v_s}(\rdd)$ in Theorem~\ref{teofinal}.  In some sense, the symbol class
$\Sigma (\rdd )$ is the largest class of symbols for which the basic
hypotheses (i) -- (iii) of Proposition \ref{fund} hold.

\end{remark}
\begin{remark}\rm
We point out a wide classes of Hamiltonian functions for which
Corollary \ref{corfinal} (but not Theorem \ref{teofinal}) holds as
well. Let $a(x,\xi)$ be a {\it smooth real-valued} function satisfying
$\partial^{\alpha} a\in L^\infty$ for $|\alpha|\geq 2$.  Then it is
proved in \cite[Corollary 7.4]{tataru} that the propagators
$U(t)=e^{itH}$ belongs for every fixed $t$ to a space of Fourier
integral operators (denoted by $FIO(\chi,s)$ in \cite{Wiener}) for
which the continuity property on modulation spaces was proved in
\cite[Theorem 3.4]{Wiener}. Therefore the results in Corollary
\ref{corfinal} remain valid for such Hamiltonians. An example is given
by $H=\Delta-V(x)$, with $V(x)$ real-valued and satisfying
$\partial^{\alpha} V\in L^\infty$ for $|\alpha|\geq 2$. The continuity
result of $e^{itH}$ for this special case is also proved in
\cite[Theorem 1.1]{kki4}.
\end{remark}

\section{Representation of operators in $FIO(\A,v_s)$ of Type I}\label{gmetapoper}
Finally  we study the question which generalized metaplectic
operators are FIOs of type I.

\begin{theorem}\label{rappresentazione}
Let $s\geq 0$ and $\A=\begin{pmatrix}  A&B\\C&D\end{pmatrix} \in
Sp(d,\R)$ with $ \det A\not=0$, and  let
$T:\cS(\rd)\to\cS'(\rd)$ be a linear continuous operator.

Then $T\in
FIO(\mathcal{A},v_s)$ if and only if $T$ is a FIO of type I, i.e.,
\begin{equation}\label{fiotipo1}
Tf(x)=\int_{\rd} e^{2\pi i \Phi\phas} \sigma\phas \hat{f}(\o)d\o
\end{equation}
with the  quadratic phase  $ \Phi(x,\eta)=\frac12  x CA^{-1}x+
\eta  A^{-1} x-\frac12\eta  A^{-1}B\eta$ and a  symbol $\sigma\in M^{\infty,1}_{1\otimes v_s}(\rdd)$.
\end{theorem}
We need the following preliminary result.
\begin{proposition}\label{aggiunta0}
Let $\sigma\in M^{\infty,1}_{1\otimes v_s}(\rdd)$, and $A,B,C$ real
$d\times d$ matrices with  $B$ invertible,  and set
$\Phi(x,\eta)=\frac{1}{2}xAx+ \eta Bx+
\frac{1}{2}\eta C\eta$. Then there exists a symbol
$\tilde{\sigma}(x,\eta) \in M^{\infty,1}_{1\otimes v_s}(\rdd)$, such
that
\begin{equation}\label{aggiunta}
\sigma^w\int e^{2\pi i\Phi(x,\eta)} \widehat{u}(\eta)\,d\eta=\int
e^{2\pi i\Phi(x,\eta)}\tilde{\sigma}(x,\eta) \widehat{u}(\eta)\,d\eta
\, .
\end{equation}
\par Explicitly,
$\tilde{\sigma}=\mathcal{U}_2\,\mathcal{U}\,\mathcal{U}_1 \sigma$,
where $\mathcal{U}_1$,  $\mathcal{U}$, $\mathcal{U}_2$ are the
isomorphisms of $M^{\infty,1}_{1\otimes v_s}$ given by $(\mathcal{U}_1
\sigma)(x,\eta)=\sigma(x,\eta+Ax)$, $(\mathcal{U}_2
\sigma)(x,\eta)=\sigma(x,B^\ast \eta)$,
$\widehat{\mathcal{U}\sigma}(\eta_1,\eta_2)=e^{\pi i
  \eta_1\eta_2}\widehat{\sigma}(\eta_1,\eta_2)$.
\end{proposition}
\begin{proof}
By replacing $A$ by $(A+A^*)/2$,  we may assume that $A$ is symmetric.
 Now, we have
\[
\sigma^w\int e^{2\pi i\Phi(x,\eta)} \widehat{u}(\eta)\,d\eta=\int e^{\pi i\eta C\eta}\sigma^w(e^{\pi i x A x+2\pi i \eta B x})\widehat{u}(\eta)\,d\eta.
\]
An explicit computation (see e.g. \cite[(14.19)]{book}), then gives
\[
\sigma^w (e^{\pi i x A x+2\pi i \eta B x})= e^{\pi i x A x}\sigma_1^w(e^{2\pi i \eta B x}),
\]
where $\sigma_1(x,\eta)=\sigma(x,\eta+Ax)$ is still in $M^{\infty,1}_{1\otimes v_s}(\rdd)$; in fact, using Lemma \ref{normeeq} the space $M^{\infty,1}_{1\otimes v_s}(\rdd)$ is easily seen to be invariant with respect to linear changes of variables. Now we can write $\sigma_1^w$ as an operator in the Kohn-Nirenberg form, $\sigma_1^w=\sigma_2(x,D)$, for the new symbol  $\sigma_2=\mathcal{U} \sigma_1$ with $\mathcal{U}$ as in the statement (see \cite[formula (14.17)]{book}). A staightforward modification of \cite[Corollary 14.5.5]{book} shows that the modulation spaces $M^{p,q}_{v_s}(\rdd)$ are invariant under the action of $\mathcal{U}$, hence $\sigma_2\in M^{\infty,1}_{1\otimes v_s}(\rdd)$. We now have
\[
\sigma^w (e^{\pi i x A x+2\pi i \eta B x})= e^{\pi i x A x}\sigma_2(x,D)(e^{2\pi i \eta B x}).
\]
On the other hand, a simple computation shows that
\[
\sigma_2(x,D)(e^{2\pi i \eta B x})= e^{2\pi i \eta B x}\sigma_2(x,B^\ast\eta),
\]
so that \eqref{aggiunta} holds with $\tilde{\sigma}(x,\eta)= \sigma_2(x,B^\ast\eta)$, which belongs to $M^{\infty,1}_{1\otimes v_s}(\rdd)$.
\end{proof}
\begin{proof}[Proof of Theorem \ref{rappresentazione}]
Assume that  $T$ is type I with  $\Phi(x,\eta)$ as in \eqref{fase} and
$\sigma\in M^{\infty,1}_{1\otimes v_s}(\rdd)$. By applying Proposition
\ref{aggiunta0} and then the representation  of $\mu (\cA
)$ as a type I FIO, it follows that
$T$ can be factorized as $\sigma^{w}\mu(\mathcal{A})$, with
$\sigma\in M^{\infty,1}_{1\otimes v_s}(\rdd)$, so that $T\in
FIO(\mathcal{A},v_s)$ by Theorem \ref{prod1}.\par
Conversely, assume  that $T\in FIO(\mathcal{A},v_s)$. According to
Theorem \ref{pseudomu},  $T$ factorizes  as  $T=\sigma^w\mu(\A)$, with
$\sigma\in M^{\infty,1}_{1\otimes v_s}(\rdd)$, and the conclusion follows by applying the FIO representation of
$\mu (\A )$  and then Proposition \ref{aggiunta0}.
\end{proof}
\par
Recall that  a {\it FIO of type II}  is the formal $L^2$-adjoint of a FIO of
type I and can be written in the  form
\begin{equation}\label{type2}
   Tf(x)=T_{II,\Phi,\sigma}(x):=\int_{\rdd}e^{-2\pi i [\Phi(y,\eta)-x\eta]}\tau(y,\eta) f(y)dy\,d\o.
\end{equation}

\begin{theorem} \label{Winerfio}
 Consider $T\in FIO(\mathcal{A},v_s)$, $s\geq 0$ and $\A=\begin{pmatrix}  A&B\\C&D\end{pmatrix}\in
Sp(d,\R)$ with $ \det A\not=0$. If $T$ is invertible on $L^2(\rd)$,  then
$T^{-1}$ is a FIO of type II with the same
phase $\Phi$ in \eqref{fase} and  a symbol $\tau\in M^{\infty,1}_{1\otimes v_s}(\rdd)$.
\end{theorem}
\begin{proof} The proof is a small modification of  \cite[Theorem
  4.6]{fio5}.  As in the proof of Theorem \ref{Wineroper}, the adjoint
  operator $T^\ast$   belongs to  $FIO(\A \inv, v_s)$ and  $P=T^*T\in
  FIO(\mathrm{Id},v_s)$.  Write
$$
T^{-1}=P^{-1} T^\ast = (T P\inv )^\ast  \, .
$$
Since $P^{-1}\in  FIO(\mathrm{Id},v_s)$ by Proposition~\ref{fund}(iii), the
algebra property of Theorem~\ref{prod1} implies  that $T P\inv
\in  FIO(\A\circ\mathrm{Id},v_s)=FIO(\A,v_s)$. By Theorem~\ref{rappresentazione}
$TP\inv $ is of  type I
with phase $\Phi $ and a symbol $\rho \in M^{\infty,1}_{1\otimes
  v_s}$. This implies
$T^{-1}\in FIO(\A\inv,v_s)$, $T\inv$ of type II with the same phase $\Phi$ and  symbol $\tau(y,\eta)=\overline{\rho(y,\eta)}$ which is still in $M^{\infty,1}_{1\otimes v_s}(\rdd)$.
\end{proof}

\end{document}